\documentclass{amsart}
\usepackage{amsthm}
\usepackage{amsmath}
\usepackage{amssymb}
\usepackage{euscript}
\usepackage{xcolor}
\usepackage{amsfonts}
\usepackage{enumitem}

\newtheorem{theorem}{Theorem}

\newtheorem{lemma}[theorem]{Lemma}

\newtheorem{dfn}[theorem]{Definition}

\newtheorem*{maintheorem*}{Main Theorem}

\newcommand{\N}{\mathbb{N}}

\newcommand{\w}{\omega}
\newcommand{\eps}{\epsilon}
\newcommand{\sh}{\sigma}
\newcommand{\ns}{\varnothing}

\newcommand{\BlockComment}[1]{}

\title{Specification and $\omega$-chaos in non-compact systems}
\author[C. Hammon]{Cordell Hammon}
\address[C. Hammon]{Department of Mathematics, Baylor University, Waco, TX 76798--7328,USA}
\email{cordell\_hammon@baylor.edu}
\author[J. Meddaugh]{Jonathan Meddaugh}
\address[J. Meddaugh]{Department of Mathematics, Baylor University, Waco, TX 76798--7328,USA}
\email{jonathan\_meddaugh@baylor.edu}
\author[J. Mohn]{Jasmin Mohn}
\address[J. Mohn]{Department of Mathematics, Baylor University, Waco, TX 76798--7328,USA}
\email{jasmin\_mohn1@baylor.edu}
\author[B. E. Raines]{Brian E. Raines}
\address[B. E. Raines]{Department of Mathematics, Baylor University, Waco, TX 76798--7328,USA}
\email{brian\_raines@baylor.edu}

\begin{document}
\subjclass[]{}
\begin{abstract}
In this paper, we demonstrate conditions under which a Lindel\"{o}f dynamical system exhibits $\omega$-chaos. In particular, if a system exhibits a generalized version of the specification property and has at least three points with mutually separated orbit closures, then the system exhibits dense $\omega$-chaos.
\end{abstract}
\maketitle
\section{Introduction}

Chaotic behavior in dynamical systems is an area of considerable mathematical interest. Fittingly, there are a large number of non-equivalent notions of chaos  including Li-Yorke chaos \cite{Li-Yorke-Chaos}, distributional chaos \cite{distributional-chaos}, Devaney chaos \cite{devaney-chaos}, and $\w$-chaos \cite{IntroductionOfOmegaChaos}.

In this paper, we focus on $\w$-chaos, which was introduced by Li in 1993 and shown to be equivalent to positive topological entropy in the context of continuous interval maps \cite{IntroductionOfOmegaChaos}.  Briefly, a system exhibits $\omega$-chaos provided there exists an uncountable $\omega$-scrambled set, i.e., an uncountable set of points, each of which has $\omega$-limit set containing non-periodic points, and each pair of which have nonempty intersection but uncountable relative complements.
%

In 2009, Lampart and Oprocha demonstrated that a large class of shift spaces exhibit $\w$-chaos \cite{Lampart_And_Oprohca_ask_a_question}. In particular they demonstrate that (weak) specification and the existence of a non-transitive orbit is enough to guarantee that a non-degenerate shift space exhibits $\omega$-chaos. Briefly, a system $(X,f)$ exhibits \emph{(weak) specification} provided that for a desired tolerance $\delta$ there exists a ``relaxation time'' $N_\delta$ such that for $a_1\leq b_1<b_1+N_\delta\leq a_2\leq b_2<\cdots<b_{n-1}+N_\delta\leq a_n\leq b_n$ ($n=2$ for weak specification) and $x_1,\ldots x_n\in X$, there exists a point $z\in X$ such that the orbits of $z$ and $x_i$ agree for iterates of $f$ between $a_i$ and $b_i$.

The specification property was first studied by Bowen \cite{Intro_of_spec} and has implications in a wide variety of settings, including the existence of certain invariant measures \cite{Sigmund}. In addition, the relationship between the specification property and chaos has been well-studied, see e.g., \cite{spec_and_dense_distro_chaos,spec_and_distro_chaos, spec_and_distro_chaos_on_cmpct_metric_spaces, wang_and_wang}.

In their 2009 paper, Lampart and Oprocha asked whether every system with the specification property exhibits $\omega$-chaos. Meddaugh and Raines demonstrated that a weak form of specification was sufficient to ensure $\omega$-chaos for shift spaces \cite{WeakSpecAndBaire}. Hunter and Raines tendered a partial answer to this question for more general systems when they showed that compact metric systems with the specification property and uniform expansion near a fixed point exhibit $\w$-chaos \cite{REEVE}. 

The main result of this paper is a natural extension of this result to Lindel\"of systems.

\begin{maintheorem*}
	Let $(X,f)$ be an expansive dynamical system such that $X$ is Lindel{\"o}f, every open set of $X$ is uncountable, $(X,f)$ has ISP, and there exist $t_0$, $t_1, s \in X$ such that $\overline{Orb(t_0)}, \overline{Orb(t_1)}, \overline{Orb(s)}$ are each a non-zero distance away from each other. Then $(X,f)$ exhibits dense $\w$-chaos.
\end{maintheorem*}

\section{Preliminaries}\label{defs}
For the purposes of this paper, we adopt the following notational conventions. We use $\w$ to denote $\N\cup\{0\}$. Given a metric space $X$  with metric $d$ and $A,B\subseteq X$, we take $d(A,B)=\inf\{d(a,b):a\in A, b\in B\}$. A space $X$ is Lindel\"of provided that every open cover of $X$ has a countable subcover.

A dynamical system is a pair $(X,f)$ consisting of a metric space $X$ and continuous surjection $f:X\to X$. For a given system $(X,f)$, $f^n$ denotes the $n$-fold composition of $f$ 
and the orbit of a point $x\in X$ is the set $Orb_f(x)=\{f^n(x):n\in\omega\}$. A point $x$ is periodic with prime period $p$ provided that $f^p(x)=x$ and $p$ is the minimal iterate for which that holds. The $\omega$-limit set of $f$ in the system $(X,f)$ is the set of accumulation points of the orbit of $x$, i.e., $\omega_f(x)=\bigcap_{n\in\omega}\overline{\{f^i(x):i\geq n\}}$. When the function $f$ is clear from context, the subscripts on $\omega_f$ and $Orb_f$ will be supressed.

%
%

The primary object of investigation in this paper is the notion of $\omega$-chaos, as defined by Li \cite{IntroductionOfOmegaChaos}.

\begin{dfn}
	Let $(X,f)$ be a dynamical system.
	
	A subset $A$ of $X$ is \textbf{$\w$-scrambled} provided that, for $x\neq y\in A$, the following hold
	\begin{enumerate}
		\item $\w(x)\cap \w(y)\neq\ns$,
		\item $\w(x)\setminus\w(y)$ is uncountable, and
		\item $\w(x)$ contains non-periodic points
	\end{enumerate}   
	
	The system exhibits \textbf{$\w$-chaos} if it contains an uncountable $\w$-scrambled set. 
\end{dfn}

As mentioned in the introduction, the specification property, as defined by Bowen, will be important in what follows. 

\begin{dfn}
	We say a dynamical system $(X,f)$ has the \textbf{specification property} (SP) if for every $\delta>0$ there exists some $N_\delta \in \N$ such that for any finite collection of points $x_1, x_2, x_3\dots x_n\in X$ and any sequence 
	$0\leq a_1\leq b_1 < a_2\leq b_2<\dots<a_n\leq b_n$ with $a_{i+1}-b_i \geq N_\delta$ there is a point $p \in X$ such that $d(f^j(p), f^j(x_i)) < \delta$ for $a_i\leq j\leq b_i$.
\end{dfn}

In the context of non-compact dynamical systems, we will make use of a generalized notion of the specification property. It is not surprising that such a generalization is appropriate---indeed similar generalizations naturally appear in the study of $\omega$-limits in non-compact settings \cite{OmegaBaire}.

\begin{dfn}
	We say a dynamical system $(X,f)$ has the \textbf{infinite specification property} (ISP) if for every $\delta>0$ there exists some $N_\delta \in \N$ such that for any countable collection of points $x_1, x_2, x_3, \ldots \in X$ and any sequence 
	$0\leq a_1\leq b_1 < a_2\leq b_2<\dots$ with $a_{i+1}-b_i \geq N_\delta$ there is a point $p \in X$ such that $d(f^j(p), f^j(x_i)) < \delta$ for $a_i\leq j\leq b_i$.
\end{dfn}

This definition may seem too specific to be applicable. But as we will see, there are many systems with ISP. It is worth nothing that, in compact systems, the two forms of specification are equivalent.

\begin{lemma}
	If $(X,f)$ has ISP then it has SP. If $X$ is compact, then the converse holds. 
\end{lemma}
\begin{proof}
	Suppose $(X,f)$ is a system with ISP. Then for any finite collection of points $x_1, x_2,\ldots, x_n$ there is a countable collection of points $y_1, y_2,\ldots$ such that $y_i = x_i$ for all $i\leq n$. There exists a point $p\in X$ that witnesses ISP with respect to $y_1, y_2,\ldots$ and thus witnesses SP with respect to $x_1, x_2,\ldots, x_n$.
	
	Suppose now that $(X,f)$ is a compact dynamical system with SP. Fix $\delta >0$ and let $N_{\delta/2}$ be given by SP. Let $x_1,x_2,\dots$ be a countable collection of points in $X$, and let $a_1\leq b_1 < a_2\leq b_2<\dots$ with $a_{i+1}-b_i \geq N_{\delta/2}$ be a sequence of natural numbers. 
	For each $n\in \N$ we can find $p_n$ so that $d(f^j(p_n), f^j(x_i)) < \delta/2$ for $a_i\leq j\leq b_i, i\leq n$. The set $\{p_n\}_{n\in\N}$ is an infinite subset of a compact space and thus has an accumulation point, $q$. 
	Since $f^i$ is continuous for all $i\in \N$ we have $\lim_{n\to\infty}d(f^i(p_n), f^i(q))\to 0$ and so $d(f^j(q), f^j(x_i)) \leq \delta/2 < \delta$ for $a_i\leq j\leq b_i$.
\end{proof}

For the purposes of this paper, we will make use of the fact that for a countable (finite or infinite) alphabet $\Sigma$, the shift space on $\Sigma$ exhibits ISP. 

\begin{dfn}
	Let $\Sigma$ be a nonempty collection of symbols. The space $\Sigma^\w$ is the space of all sequences $(x_0, x_1, x_2,\ldots)$ where $x_i\in \Sigma$ for all $i\in\w$. 
	The space is equipped with a map $\sh$ called \textbf{the shift map} where $\sh(x_0, x_1, x_2,\ldots) = (x_1, x_2, x_3,\ldots)$.
	
	In the case that $\Sigma$ is a collection of numbers, $\Sigma^\w$ is equipped with a metric. $$d((x_0, x_1,\ldots), (y_0, y_1,\ldots)) = \sum_{i=0}^\infty\frac{\min\{|x_i-y_i|, 1\}}{2^i}$$
\end{dfn}

Of particular import going forward is the shift map with $\Sigma=2=\{0,1\}$, i.e. $2^\w = \{(x_0, x_1,\ldots) : x_i\in \{0,1\}\text{ for all } i\in \w\}$ due to the fact that Lampart and Oprocha demonstrated that it exhibits $\w$-chaos \cite{Lampart_And_Oprohca_ask_a_question}.

\begin{lemma}
    The shift space $\Sigma ^\w$ with the shift map has ISP if $|\Sigma|\geq 2$.
\end{lemma}
\begin{proof}
    Let $\delta > 0$ and let $N$ be such that $2^{-N}<\delta/2$. Let $x_1, x_2, \dots$ be a countable collection of points in $\Sigma^\N$, and let $a_1\leq b_1 <a_2\leq b_2 < a_3\leq b_3<\dots$ be a sequence of natural numbers such that $a_i - b_{i-1} > N$. 
    
    Finally, let $y$ be a point in $\Sigma^\N$ such that $y_{[a_i, b_i+N]} = (x_i)_{[a_i, b_i+N]}$. Note that such points exist since $a_i - b_{i-1} > N$. Then $\sh^j(y)$ and $\sh^{j}(x_i)$ agree for at least $N$ symbols when $a_i\leq j\leq b_i$, so $d(\sh^j(y), \sh^{j}(x_i)) < 2^{-N} < \delta$ when $a_i\leq j\leq b_i$ as required for ISP.  
\end{proof}

%

In order to make discussing the patterns of the orbit of a point given by specification easier, we introduce the following lemma.
\begin{lemma} \label{spec pattern}
	Let $(X,f)$ be a dynamical system with infinite specification and fix $\delta >0$. Let $N_\delta>0$ witness specification. For $z_0, z_1,\ldots \in X$, $c_0,c_1,\ldots \in \omega$, and $M\geq N_\delta$, there exists a $p\in X$ such that the $f^i(p)$ is less than $\delta$ away from the $i$-th term in the following sequence whenever that term is not $*$. (We use the convention that $(*)^M$ denotes $M$-many consecutive entries which are $*$.)	
	\begin{align*}z_0, f(z_0), \dots, f^{c_0}(z_0), &(*)^{M}, z_1, f(z_1), \dots, f^{c_1}(z_1), (*)^{M},\ldots\\
		 \dots, &(*)^{M},z_n, f(z_n), \dots, f^{c_n}(z_n), (*)^{M}, \ldots\end{align*}
\end{lemma}

\begin{proof}
	Fix $z_0, z_1,\ldots \in X$, $c_0,c_0,\ldots \in \omega$, and $M\geq N_\delta$. We will apply the infinite specification property to the sequence $x_0,x_1,\ldots$ and  $0\leq a_0\leq b_0<a_1\leq b_1\dots$ defined as follows. Fix $a_0=0$ and $b_0=c_0$. If $b_i$ is defined, choose $a_{i+1}=b_i+M\geq b_i+N_\delta$. If $a_i$ is defined, choose $b_i=a_i+c_i$. By surjectivity, choose $x_i\in f^{-a_i}(z_i)$. By infinite specification property, choose a point $p\in X$ such that $d(f^j(p),f^j(x_p))<\delta$ when $a_i\leq j\leq b_i$. It is immediately clear that $p$ satisfies the conclusion of the lemma.
\end{proof}
If $(X,f)$ has only the standard specification property, a similar result holds for any finite lists $x_0,\ldots x_n$ and $c_0,\ldots c_n$.

Our main result requires a form of local expansion.

\begin{dfn}
    We say $f$ is \textbf{expansive} if there is some $\eta > 0, \lambda>1$ such that if $0< d(x, y) < \eta$ then $d(f(x), f(y))> \lambda d(x,y)$  for all $x, y \in X$. 
    
\end{dfn}

The following lemma will be used frequently in the following section.

\begin{lemma}\label{Expansivity_Implies_Closeness}
	Let $(X,f)$ be expansive and let $\eta > 0, \lambda >1$ witness expansivity. 
	If $p, q\in X, j\in \N$ are such that $d(f^i(p), f^i(q)) < \eta$ for all $i\leq j$, then $d(p ,q) < \eta\lambda^{-j}$. 
	Moreover, if $p \in X$ and $\{q_n\}_{n\in \N}$ is a sequence of points such that $d(f^i(p), f^i(q_n)) < \eta$ for all $i\leq j_n$ where $j_n\to \infty$. Then $q_n \to p$.
	
\end{lemma}
\begin{proof}
	Since $d(f^i(p), f^i(q)) < \eta$ for all $i\leq j$ we have $\eta > d(f^j(p), f^j(q)) > \lambda d(f^{j-1}(p), f^{j-1}(q)) > \lambda^2 d(f^{j-2}(p), f^{j-2}(q)) > \dots > \lambda^j d(p, q).$ Dividing both sides by $\lambda^j$ yields the result. 
	
	From here, $d(f^{j_n}(p), f^{j_n}(q_n)) < \eta$ implies $d(p, q_n) < \eta\lambda^{-j_n}$. Since $j_n\to \infty$ we have $q_n\to p$.
\end{proof}

It is worth noting that $\w$-chaos is a rather local property. There are cases where a dynamical system $(X, f)$ exhibits $\w$-chaos even though the only ``chaotic" part of $X$ is a measure theoretically small portion of the space. 
Consider a finite measure system $(X,f)$ with $\w$-chaos and affix to it a much larger space $X'$ and extend $f$ to $X'$ in such a way that $f|X'$ is the identity.
The new system $(X\cup X', f)$ still has $\w$-chaos, but the chaos is relegated to the comparatively small portion of the system, $X$. To consider a more global form of $\w$-chaos we introduce the following. 

\begin{dfn}
	We say a system exhibits \textbf{dense $\w$-chaos} if any open set contains an uncountable $\w$-scrambled set. 
\end{dfn}

 Systems with dense $\w$-chaos do exist. In fact, $2^\w$ has dense $\w$-chaos 
 To see this, let $\alpha\in 2^\w$. Note that prepending any finite word $w$ to an infinite word $\alpha$ does not alter the $\w$-limit set of $\alpha$. 
 That is to say $\w(w\alpha) = \w(\alpha)$. 
 Note, also, that basic open sets in $2^\w$ take the form of cylinder sets. So pick any finite word $w\in 2^\w$. Prepend every point in $2^\w$ with $w$ to get the open set $[w]\subseteq 2^\w$. This open set exhibits $\w$-chaos.

\section{Existence of $\omega$-chaos in certain Lindel{\"o}f systems}
In this section, we will prove the following theorem.
\begin{maintheorem*}\label{MainThm}
Let $(X,f)$ be an expansive dynamical system such that $X$ is Lindel{\"o}f, every open set of $X$ is uncountable, $(X,f)$ has ISP, and there exist $t_0$, $t_1, s \in X$ such that $\overline{Orb(t_0)}, \overline{Orb(t_1)}, \overline{Orb(s)}$ are each a non-zero distance away from each other. Then $(X,f)$ exhibits dense $\w$-chaos.
\end{maintheorem*}

We begin with the following constructions.
Let $(X,f)$ be given as in the main theorem. Let $\eta, \lambda$ witness the expansivity of $f$. Let $t_0, t_1, s$ be three points 
such that the orbit closures of these points have non-zero distance between each other.
Fix $\xi$ in $X$ and $U$ a neighborhood thereof. Let $D>0$ be small enough so that the ball of radius $D$ around $\xi$ is contained entirely within $U$.
Fix $\epsilon > 0$ such that $$\min\{d(\overline{Orb(t_o)}, \overline{Orb(t_1)}), d(\overline{Orb(t_o)}, \overline{Orb(s)}), d(\overline{Orb(t_1)}, \overline{Orb(s)}), D, \eta\}> 2\epsilon.$$ 
Lastly, let $N$ be given by ISP with respect to $\epsilon / 2$
and let $P\in \N$ such that 
$P> N$. 
We construct two sequences of numbers as follows: $a_0 = 0$, if $a_i$ is defined, let $b_i = a_i + P$, and if $b_i$ is defined, let $a_{i+1} = b_i+N$.

With that construction done, we are now ready to leverage ISP.
\begin{dfn}
    For $\beta \in 2^\omega$, let $E_\beta = \{x\in X : a_i\leq j\leq b_i, d(f^j(x), f^{j-a_i}(t_{\beta_i}))\leq \frac{\epsilon}{2}\text{ for all }i\in\w\}$. 
\end{dfn}

Note that $E_\beta \neq \varnothing$ due to ISP. 

\begin{lemma}
    $E_\beta$ is closed, and $E_\beta\cap E_\gamma \neq \varnothing$ if and only if $\beta = \gamma$.
\end{lemma}
\begin{proof}
    Let $\{p_k\}_{k\in\w}$ be a sequence of points in $E_\beta$ which limits to $p$. Choose $i\in \w$ and choose some $j\in \{a_i,\dots, b_i\}$. 
    Since $f^j$ is a continuous function, $f^j(p_k) \to f^j(p)$ and thus $d(f^j(p), f^{j-a_i}(t_{\beta_i}))\leq \epsilon/2$. So $p\in E_\beta$.
    
    To verify the other claim, let $p\in E_\beta \cap E_\gamma$. Then for any $j\in\{a_i,\dots,b_i\}$ we have that $d(f^j(p), f^{j-a_i}(t_{\beta_i}))\leq \epsilon/2$ and $d(f^j(p), f^{j-a_i}(t_{\gamma_i}))\leq \epsilon/2$. But, since $d(Orb(t_0), Orb(t_1)) > 2\epsilon$, we must have $t_{\beta_i}=t_{\gamma_i}$ for all $i$ and thus $\beta=\gamma$.
\end{proof}

\begin{lemma}\label{E_beta_not_periodic}
    If $\beta\in 2^\w$ is not periodic, then $E_\beta$ does not contain periodic points. 
\end{lemma}
\begin{proof}
   Suppose $q\in E_\beta$ is periodic with period $K$. Then $d(f^j(q), f^{j-a_i}(t_{\beta_i}))\leq\eps/2$ for all $i\in \w, a_i\leq j\leq b_i$.
    Note that $a_i = i(N+P)$ 
    so $a_i+ a_h = i(N+P) + h(N+P) = (i+h)(N+P) = a_{i+h}$.   
    Let $L = lcm(K, N+P)$ and choose $m$ such that $L=m(N+P)$ and note that $L=a_m$.
    Moreover, $f^{L+j}(q) = f^j(q)$, and thus we have the following.
    \begin{align*}
         d(t_{\beta_{m+i}}, t_{\beta_i}) &\leq d(t_{\beta_{m+i}}, f^{a_i}(q)) + d(f^{a_i}(q), t_{\beta_i})\\
         &=d(t_{\beta_{m+i}}, f^{L+a_i}(q)) + d(f^{a_{i}}(q), t_{\beta_i}) \\
         &=d(t_{\beta_{m+i}}, f^{a_m+a_i}(q)) + d(f^{a_{i}}(q), t_{\beta_i}) \\
         &=d(t_{\beta_{m+i}}, f^{a_{m+i}}(q)) + d(f^{a_{i}}(q), t_{\beta_i}) \\
         &\leq \eps
        \nonumber
    \end{align*}

    Since $Orb(t_{\beta_0})$ and $Orb(t_{\beta_1})$ are $2\eps$ apart this means that $\beta_{i+m} = \beta_i$ for all $i\in\w$, i.e. $\beta$ is periodic. 
\end{proof}

\begin{lemma}\label{we_can_spec_images_of_beta}
    Let $\beta\in 2^\w$ and let $q\in E_\beta$. Then $f^{a_k}(q) \in E_{\sh^k(\beta)}$. 
\end{lemma}
\begin{proof}
    Fix $k\in \w$.
    By definition of $E_\beta$, $d(f^j(q), f^{j-a_i}(t_{\beta_i}))\leq \frac{\epsilon}{2}, a_i\leq j\leq b_i\text{ for all }i\in\w$ and so it certainly holds for all $i>k$. So $d(f^j(f^{a_k}(q)), f^{j-a_i}(t_{\sh^k(\beta)_i}) \leq \eps/2$. 
\end{proof}
Going forward, the following sets will be useful \[G_\beta = {\bigcup_{i\in \w} f^i(E_\beta)}, \quad H_\beta' = {\bigcup_{\alpha\in \w_\sh(\beta)}G_\alpha}, \quad H_\beta = \overline{\bigcup_{\alpha\in \w_\sh(\beta)}G_\alpha}\]

The definition of $H_\beta$ leads to the following useful lemmas which will be used to establish that certain $\w$-limit sets have uncountable set differences. The first lemma is immediate from the definitions.

\begin{lemma}\label{H_beta_sets_are_uncountable}
    If $\beta\in 2^\w$ is such that $\w(\beta)$ is uncountable, then $H_\beta'$ and $H_\beta$ are uncountable. 
\end{lemma}

	

Less immediately, we have the following.
\begin{lemma}\label{b_not_in_orb_closure_implies_Eb_not_meet_H}
	If $\beta,\chi\in 2^\w$ are such that $\beta\notin \overline{Orb(\chi)}$, then $E_\beta\cap H_\chi = \varnothing$.
\end{lemma}
\begin{proof}
	Let $\beta, \chi\in 2^\w$ such that $\beta\notin \overline{Orb(\chi)}$. Let $r$ be such that $\beta_{[0, r]}$ does not occur in $\chi$. 
	Suppose, by way of contradiction, that $\{z_k\}_{k\in \w}$ is a sequence of points in $H_\chi'$ which limits to $z\in E_\beta$. Then for all but finitely many $z_k$ we have, for $j\leq b_r$, 
	\begin{equation}\label{zk_and_z_are_close}
		d(f^j(z_k), f^{j}(z)) \leq \eps/2
	\end{equation}

	
	Since each $z_k$ belongs to  $H_\chi'$, for each $k$ there exists $v_k\in \w$ minimal such that $z_k\in f^{v_k}(E_{\alpha_k})$ for some $\alpha_k$. As such, by Lemma \ref{we_can_spec_images_of_beta}, for each $i\in \w$ with $a_i > v_k$ we have $f^{a_i-v_k}(z_k)\in E_{\sh^i(\alpha_k)}$. 
	For each $k$, let $i_k$ be minimal so that $a_{i_k}\geq v_k$ and let $q_k = a_{i_k}-v_k.$
	Then $f^{q_k}(z_k)\in E_{\theta^k}$ for some $\theta^k\in \w(\chi)$.
	
	Notice that by construction of the $E_\alpha$ sets, any point in $E_\alpha$ is exactly $P+N$ many iterates away from being in $E_{\sh(\alpha)}.$ So any point in $H_\chi'$ is fewer than $P+N$ iterates away from being in some $E_\alpha$ (if it take $N+P$ many iterates, then we must have started out in some $E_{\sh^{-1}(\alpha)}$ set and thus needed 0 iterates to be in an $E$ set). 
	Moreover, for any point in $E_\alpha$, the first $P$ many iterates are controlled by specification, and iterates number $P+1, P+2, \dots, P+N-1$ are all not controlled by specification.
	So, if $q_k < N$, then  $z_k$ is in a part of the orbit of $E_{\alpha^k}$ that is not controlled by specification.
	Similarly, if $q_k \geq N$, then $z_k$ is in a part of the orbit of $E_{\alpha^k}$ that \emph{is} controlled by specification.
	As such we have two cases. 
	
	If $q_k < N$, let $y_k = q_k$ and let $\delta^k = \theta^k$. So $f^{y_k}(z_k)\in E_{\delta^k}$.
	If, on the other hand, $q_k \geq N$, then since $z_k$ is in a part of the orbit of $E_{\alpha^k}$ that is controlled by specification, we let $y_k = 0$ and let $\delta^k = \alpha^k$. 
	
	Thus, in either case we have $f^{a_i}(f^{y_k}(z_k)) = f^{a_i+y_k}(z_k)\in B_{\eps/2}(Orb(t_{\delta^k_i}))$ for all $i\in \w$. 
	
	But, $\beta_{[0, r]}$ does not occur in $\chi$, so for each $k$ there is some index $i_k \leq r$ such that $\beta_{i_k}\neq \delta^k_{i_k}$ and so $f^{a_{i_k}}(f^{y_k}(z_k)) \in B_{\eps/2}(t_{1-\beta_{i_k}})$.
	
	Since $z\in E_\beta$, we have $f^{a_{i_k}}(z) \in B_{\eps/2}(t_{\beta_{i_k}})$ for each $k\in \w$, but more importantly $f^{a_{i_k}+y_k}(z)\in B_{\eps/2}(f^{y_k}(t_{\beta_{i_k}}))$ since $y_k \leq N$ and so $a_{i_k}+y_k<b_{i_k}.$ 
	
	But this, combined with (\ref{zk_and_z_are_close}), means $f^{a_{i_k}+y_k}(z)\in B_{\eps}(Orb(t_{\beta_{i_k}})) \cap B_{\eps}(Orb(t_{1-\beta_{i_k}}))$, which is empty, a contradiction.
\end{proof}

We now demonstrate how, given $\beta$ and $\chi$ in $2^\omega$, to construct a pair of points whose $\w$-limit sets meet and will each contain one of $H_\beta$ or $H_\chi$ and miss the other.



For $n\in \w, z\in H_\beta$ let $U_{z,n} = \{ y: d(f^i(y), f^i(z)) <\eps/4\text{ for }i\leq b_n\}$. Then $\mathcal{U}_n = \{U_{z,n}:z\in H_\beta\}$ is an open cover of $H_\beta$. 
$H_\beta$ is closed subset of the Lindel{\"o}f space $X$ and is therefore Lindel{\"o}f \cite{Munkres}. 
Thus, we can choose a countable set $\{U_{z_1, n}, U_{z_2,n},\dots\}$ which is an open cover of $H_\beta$. Finally, let $H_{\beta, n} = \{z_1, z_2, z_3,\dots\}$. Since each $H_{\beta, n}$ is countable, we can enumerate the sets:
\begin{align*}
H_{\beta, 1} &= \{\gamma_{1, 0}, \gamma_{1, 1}, \gamma_{1,2},\dots\}\\
H_{\beta, 2} &= \{\gamma_{2, 0}, \gamma_{2, 1}, \gamma_{2,2},\dots\}\\
H_{\beta, 3} &= \{\gamma_{3, 0}, \gamma_{3, 1}, \gamma_{3,2},\dots\}
\end{align*}
We can list the elements of the union of the $H_{\beta, i}$ sets via the following diagonalization: $$(\gamma_{1,0}, \gamma_{1,1}, \gamma_{2,0}, \gamma_{1,2}, \gamma_{2,1}, \gamma_{3,0}, \dots)$$
Since the collection of sets $H_{\beta ,n}$ are not necessarily disjoint, the previous list may have repetition, but that does not affect the proof. 

Let $M$ be given by ISP with respect to $\epsilon/8$. Now, apply Lemma \ref{spec pattern} to find a point $p_\beta$ such that $f^i(p_\beta)$ is less than $\epsilon/8$ from the $i$-th term the sequence 
\begin{align*}
     \xi, (*)^M, & \gamma_{1,0},\dots, f^{b_1}(\gamma_{1,0}), (*)^M, s, (*)^M, \\
    & \gamma_{1,1},\dots, f^{b_1}(\gamma_{1,1}), (*)^M, s, f(s), (*)^M,\\
    & \gamma_{2,0},\dots, f^{b_2}(\gamma_{2,0}), (*)^M, s, f(s), f^2(s), (*)^M,\\ & \dots.
\end{align*}
whenever the $i$-th term is not a $*$.
Note that the sequence follows the orbit of the point $\gamma_{j, w}$ for $b_j$-many iterates, and we follow the orbit of $s$ for $(n-1)$-many iterates $n-$th time we arrive at $s$ in the sequence.

We now turn our attention to the $\omega$-limit set of $p_\beta$.

\begin{lemma}
    $H_\beta\cup \{s\} \subseteq \w(p_\beta)$. 
\end{lemma}
\begin{proof}
    To see that $s\in \w(p_\beta)$, use the specification pattern of $p_\beta$ to define a sequence of integers $i_k$ such that $f^{i_k}(p_\beta) \in B_{\eps/8}(s)$ and $f^{i_k+j}(p_\beta)\in B_{\eps/8}(f^j(s))$ for all $j\leq k, k\in \w$. Lemma \ref{Expansivity_Implies_Closeness} then gives us that $\lim_{k\to \infty}f^{i_k}(p_\beta) = s$.
    
    Now to see that $H_\beta$ is contained in $\w(p_\beta)$, let $\upsilon \in H_\beta$. For each $n\in \N$ there exists an $l_n$ such that $\upsilon \in U_{\gamma_{n, l_n}, n}$. 
    Notice that since $\upsilon\in U_{\gamma_{n, l_n}, n}$ we have that $d(f^i(\upsilon), f^i(\gamma_{n,l_n}))<\eps/4$ for all $i\leq b_n$. Thus, by the expansivity of $f$,
    together with Lemma \ref{Expansivity_Implies_Closeness}
    we have that $\lim_{n\to \infty}\gamma_{n,l_n} = \upsilon$.
    
    Now define a sequence $i_n$ such that $f^{i_n}(p_\beta)\in B_{\eps/8}(\gamma_{n, l_n})$ and for all $j\in \{0, 1, \dots, b_{n}\}$ we have $f^{i_n+j}(p_\beta)\in B_{\eps/8}(f^{j}(\gamma_{n, l_n}))$. Such a sequence exists by the construction of $p_\beta$ via ISP. 
    Then, 
    \[
        d(f^{i_n}(p_\beta), \upsilon) \leq d(f^{i_n}(p_\beta), \gamma_{n, l_n}) + d(\gamma_{n, l_n} , \upsilon).
    \]
    But, as we have already seen $d(\gamma_{n, l_n} , \upsilon)\to 0$ as $n\to \infty$. Moreover, by Lemma \ref{Expansivity_Implies_Closeness}, since $$d(f^{i_n+j}(p_\beta), f^{j}(\gamma_{n, l_n})) <\eta$$ for all $j<b_n$, we have $d(f^{i_n}(p_\beta), \gamma_{n, l_n}) < \eta\lambda^{-b_n}$. So as $n\to\infty, d(f^{i_n}(p_\beta), \gamma_{n, l_n}) \to 0$. 
\end{proof}

\begin{lemma}\label{omega_p_gamma_is_disjoint_from_H_beta}
	If $\beta, \gamma\in2^\w$ are such that $\overline{Orb(\beta)}\cap \overline{Orb(\gamma)} = \varnothing$, then $\w(p_\gamma) \cap H_\beta' = \varnothing$.
\end{lemma}
\begin{proof}
	Let $q\in H_\beta'$. Then $q\in\bigcup _{\alpha \in \w(\beta)} G_\alpha$.
	As such, be Lemma \ref{we_can_spec_images_of_beta} there is some $m\in\w$ such that $f^m(q)\in E_{\alpha}$ for some $\alpha\in \w(\beta)$.  
	By definition of $E_\alpha$ we have $d(f^{j}(f^m(q)), f^{j-a_i}(t_{\alpha_i}))\leq \eps/2$ for all $j\in\{a_i,\dots, b_i\}, i\in\w$.
	
	
	Now suppose, by way of contradiction, $q\in \w(p_\gamma)$. This implies that $f^m(q)\in \w(p_\gamma)$ as well since $\w$-limit sets are forward invariant. 
	Let $\{k_n\}_{n\in\w}$ be a sequence of natural numbers such that $\lim_{n\to\infty}f^{k_n}(p_\gamma) = f^m(q)$.

	
	By the construction of $p_\gamma$, for every $k_n$, $f^{k_n}(p_\gamma)$ is either 
	\begin{enumerate}
		\item in a part of $Orb(p_\gamma)$ that is controlled by specification and within $\eps/8$ of $\xi$ (but this only occurs once at the start of the spec pattern of $p_\gamma$), 
		\item in a part of $Orb(p_\gamma)$ that is controlled by specification and within $\eps/8$ of the orbit of $s$,
		\item in a part of $Orb(p_\gamma)$ that is controlled by specification and within $\eps/8$ of the orbit of some $\gamma_{i, j}$, or
		\item not in a part of $Orb(p_\gamma)$ that is controlled by specification, in which case $f^{k_n}(p_\gamma)$ is fewer than $M+1$ iterates away from being within $\eps/8$ of some $\gamma_{i,j}$ or $s$.
	\end{enumerate}
	For the sake conciseness, let $p_n = f^{k_n}(p_\gamma)$. Observe that at least one of conditions (1)-(4) must be met infinitely often, and since condition (1) is only satisfied once, at least one of conditions (2)-(4) must be met infinitely often.

	Suppose condition (2) is satisfied by infinitely many $p_n$. Then $f^m(q) = \lim_{n\to\infty} p_n$ would be within $\eps/4$ of $Orb(s)$. But this cannot be since, by choice of $\epsilon$, we have $d(\overline{Orb(s)}, \overline{Orb(t_{\alpha_0})}) > 2\eps$, and $f^m(q)$ is within $\eps/2$ of $t_{\alpha_0}$ by virtue of being in $E_\alpha$.

	Suppose condition (3) is satisfied by infinitely many $p_n$. Then for infinitely many $n$, $p_n$ is within $\eps/8$ of some point $\gamma'_{i_n, j_n}$ in the orbit of some $\gamma_{i_n, j_n}$. Let $L_n$ be minimal so that $L_n +k_n = b_i +1$ for some $i\in \w$.
	Then, $p_n$ is guaranteed to be within $\eps/8$ of the orbit of $\gamma_{i_n, j_n}$ for at least $L_n$ many iterates. 
	
	We have two possibilities. Either $\{L_n\}_{n\in \w}$ is unbounded or not.  If it is unbounded, then we can pass to $\{p_{n_r}\}_{r\in\w}$, a subseqeunce of $p_n$ on which $L_{n_r}$ is monotone increasing.
	Then by Lemma \ref{Expansivity_Implies_Closeness}, we have that $d(p_{n_r}, \gamma'_{i_{n_r}, j_{n_r}})\to 0$. Moreover, since $p_{n_r}\to f^m(q)$, we have $d(p_{n_r}, f^m(q)) \to 0$.
	
	So, by the triangle inequality
	\[
		d(f^m(q), \gamma'_{i_{n_r}, j_{n_r}}) \leq d(f^m(q), p_{n_r}) + d(p_{n_r}, \gamma'_{i_{n_r}, j_{n_r}})
	\]
	The latter two terms go to 0 as $n_r\to \infty$ and so $d(f^m(q), \gamma'_{i_{n_r}, j_{n_r}})\to 0$. This means that $\lim_{r\to\infty} \gamma'_{i_{n_r}, j_{n_r}} = f^m(q)$. 
	But each $\gamma'_{i_{n_r}, j_{n_r}} \in H_\gamma$ since each $\gamma'_{i_{n_r}, j_{n_r}}$ is in the orbit of some $\gamma_{i,j}\in H_\gamma$ and $H_\gamma$ is forward invariant. 
	Moreover, $H_\gamma$ is closed, which implies that $f^m(q)\in H_\gamma$. 
	But, $E_\alpha \cap H_\gamma = \varnothing$ by Lemma \ref{b_not_in_orb_closure_implies_Eb_not_meet_H}, a contradiction.
	
	If, on the other hand, $\{L_n\}_{n\in\w}$, is bounded, then we can pass to $\{p_{n_r}\}_{r\in\w}$, a subseqeunce of $p_n$, on which $L_{n_r}$ is constant.
	Let $L$ be this constant. Then $f^{M+L}(p_n)$ is within $\eps/8$ of $s$ for all $r$. 
	Additionally, every time an iterate of $p_\gamma$ in the specification pattern is within $\eps/8$ of $s$, it will stay within $\eps/8$ of $Orb(s)$ for increasingly more iterates. 
	Hence as $n_r\to \infty$ increasingly more iterates of $f^{M+L}(p_{n_r})$ will stay within $\eps/8$ of $Orb(s)$. 
	This fact, in conjunction with Lemma \ref{Expansivity_Implies_Closeness}, implies that the sequence $\{p_{n_r}\}_{r\in\w}$ is actually converging to $s$. 
	But, the forward orbit of $f^m(q)$ will be within $\eps/2$ of $\{t_0, t_1\}$ infinitely often.
	So, since $d(\overline{Orb(s)}, \overline{Orb(t_{i})}) > 2\eps$ for $i\in \{0,1\}$, we cannot have $\{f^{M+L}(p_n)\}_{n\in\w}$  converging to $s$ and to $f^{M+L}(f^m(q))$. 
	
	Lastly, suppose that for infinitely many $p_n$ condition (4) is satisfied. If infinitely many $p_n$ are not in a part of $Orb(p_\gamma)$ controlled by specification and are fewer than $M+1$ iterates away from being within $\eps/8$ of $s$, then there is some $w\leq M+1$ such that infinitely many $p_n$ are exactly $w$ iterates away from being within $\eps/8$ of $s$. 
	By construction of $p_\gamma$, as $n\to \infty$, increasingly more iterates of $f^w(p_n)$ will be within $\eps/8$ of $s$ which implies $f^{w}(p_n)\to s$ by Lemma \ref{Expansivity_Implies_Closeness}. 
	This would mean $f^m(q) = s$, a contradiction. 
	
	On the other hand, if infinitely many $p_n$ are fewer than $M+1$ iterates away from being within $\eps/8$ of some $\gamma_{i_n,j_n}$, then there is some $w\leq M+1$ such that infinitely many $p_n$ are exactly $w$ iterates away from being within $\eps/8$ of $\gamma_{i_n,j_n}$. By construction of the orbit of $p_\gamma$,  this means that $d(f^j(f^w(p_n)), f^j(\gamma_{i_n, j_n})) < \eps/8$ for $j\leq b_{i_n}$. 
	
	Again, we have two possibilities. Either the collection $\{i_n\}_{n\in\w}$ is unbounded or bounded. If the collection is unbounded, then we can pass to $\{p_{n_r}\}_{r\in\w}$, a subseqeunce of $\{p_n\}_{n\in\w}$ on which $\{i_{n_r}\}_{r\in\w}$ is monotone increasing and $f^w(p_{n_r})$ is within $\eps/8$ of $\gamma_{i_{n_r}, j_{n_r}}$.
	Since $i_{n_r}$ is monotone increasing, $f^w(p_{n_r})$ will be within $\eps/8$ of $Orb(\gamma_{i_{n_r}, j_{n_r}})$ for increasingly many iterates, which by Lemma \ref{Expansivity_Implies_Closeness} means $d(f^w(p_{n_r}), \gamma_{i_{n_r}, j_{n_r}})\to 0$. As before, this means $d(f^w(f^m(q)), \gamma_{i_n, j_n}) \to 0$. 
	Now, let $g\in \w$ be such that $f^{g}(f^w(f^m(q)))\in E_\delta$ for some $\delta\in \w(\beta)$. Such a $g$ exists by Lemma \ref{we_can_spec_images_of_beta}.
	 
	Since $f^g$ is continuous, we have $\lim_{r\to\infty}f^g(\gamma_{i_{n_r}, j_{n_r}}) = f^g(f^w(f^m(q)))$. 
	But, $H_\gamma$ is forward invariant, so $f^g(\gamma_{i_n, j_n})\in H_\gamma$, and $H_\gamma$ is closed, so $f^g(f^w(f^m(q)))\in H_\gamma$.
	But, $\delta\notin \overline{Orb(\gamma)}$ so $E_\delta\cap H_\gamma = \varnothing$, a contradiction.
	
	If, on the other hand $\{i_n\}_{n\in \w}$ is bounded, we can pass to $\{p_{n_r}\}_{r\in\w}$, a subsequence of $\{p_n\}_{n\in\w}$, such that $\{i_{n_r}\}_{r\in\w}$ is constant. 
	Let $L$ be this constant. Then $f^{M+b_L+w}(p_{n_r})$ is within $\eps/8$ of $s$. And again, since we stay within $\eps/8$ of the orbit of $s$ for increasingly many iterates each time we visit in the specification pattern of $p_\gamma$, this means $f^{M+b_L+w}(p_{n_r}) \to s$. But as before, since inifnitely many iterates of $f^m(q)$ are within $\eps/2$ of $\{t_0, t_1\}$ and $d(\overline{Orb(s)}, \overline{Orb(t_{\alpha_i})}) > 2\eps$ for $i\in \{0,1\}$, we cannot have $f^{M+b_L+w}(p_{n_r})$ converging to $s$ and $f^{M+b_L+w}(f^m(q))$. 
	
	Thus, $f^m(q)\notin \w(p_\gamma)$, implying $q\notin \w(p_\gamma)$ which is the desired contradiction. 
\end{proof}

We are now ready to prove our main theorem.

\begin{maintheorem*}
	Let $(X,f)$ be an expansive dynamical system such that $X$ is Lindel{\"o}f, every open set of $X$ is uncountable, $(X,f)$ has ISP, and there exist $t_0$, $t_1, s \in X$ such that $\overline{Orb(t_0)}, \overline{Orb(t_1)}, \overline{Orb(s)}$ are each a non-zero distance away from each other. Then $(X,f)$ exhibits dense $\w$-chaos.
\end{maintheorem*}

\begin{proof}
	Fix $\xi$ in $X$ and $U$ a neighborhood thereof. Let $W\subset 2^\w$ be an uncountable set of non-periodic points with pairwise disjoint orbits and uncountable $\w$-limit sets. Such a set does exist \cite{uncountable_minimal_sets_in_2w}. 
	Let $\beta, \gamma \in W$. Find $p_\beta$ and $p_\gamma$ in $X$ following the method outlined above. By Lemma \ref{E_beta_not_periodic}, since $\beta, \gamma$ are not periodic, $E_\beta$ and $E_\gamma$ do not contain periodic points implying that $H_\beta$ and $H_\gamma$ do not contain \emph{only} periodic points. 
	Moreover, $\w(p_\beta)\cap \w(p_\gamma)\ni s$, and $\w(p_\beta)\setminus \w(p_\gamma) \supset H_\beta'$ by Lemma \ref{omega_p_gamma_is_disjoint_from_H_beta} which is uncountable as shown in Lemma \ref{H_beta_sets_are_uncountable}.
	Similarly, $\w(p_\gamma)\setminus \w(p_\beta) \supset H_\gamma'$ which is also uncountable.
	There are uncountably many such $\beta$ and $\gamma$ in $W$. From this, it follows that $(X,f)$ exhibits $\w$-chaos.
	
	Moreover, $p_\beta, p_\gamma \in B_{\eps/8}(\xi) \subset U$. Thus $(X,f)$ exhibits dense $\w$-chaos.
\end{proof}
Note that if $X$ contains some open sets which are not uncountable, then $X$ will still exhibit $\w$-chaos, though not densely. 

\bibliographystyle{plain}
\bibliography{OmegaChaosBib}
\end{document}